\documentclass[letter,12pt]{article}

\usepackage{amssymb}
\usepackage{amsthm}
\usepackage{amsfonts}
\usepackage[latin2]{inputenc}
\usepackage{amsmath}
\usepackage{anysize}
\usepackage{hyperref}
\usepackage{bbm}
\usepackage{enumitem}
\marginsize{2.6cm}{2.6cm}{1cm}{2cm}

\usepackage{titlesec}
\titleformat{\section}
  {\normalfont\large\bfseries}
  {\thesection}{1em}{}
\titleformat{\subsection}
  {\normalfont\bfseries}
  {\thesubsection}{1em}{}

\newtheorem*{theorem*}{Theorem}
\newtheorem{theorem}{Theorem}[section]

\newtheorem{lemma}[theorem]{Lemma}
\newtheorem{corollary}[theorem]{Corollary}

\newtheorem{proposition}[theorem]{Proposition}

\def\K{K\"ahler }

\def\del{\partial}
\def\dbar{\bar\partial}
\def\ddbar{\del\dbar}
\def\ra{\rightarrow}

\newcommand{\RR}{\mathbb{R}}

\def\del{\partial}

\DeclareMathOperator{\Ric}{Ric}

\def\o{\omega}

\def\PSH{\textup{PSH}}

\title{Metric geometry of normal K\"ahler spaces, energy properness, and existence of canonical metrics \vspace{-0.1in}}

\author{Tam\'as Darvas}
\date{\vspace{-0.2in}}

\begin{document}
\maketitle

\begin{abstract}
Let $(X,\o)$ be a compact normal K\"ahler space, with Hodge metric $\o$. In this paper, the last in a sequence of works studying the relationship between energy properness and canonical K\"ahler metrics, we introduce a geodesic metric structure on $\mathcal H_{\o}(X)$, the space of K\"ahler potentials,  whose completion is the finite energy space $\mathcal E^1_{\o}(X)$. Using this metric structure and the results of Berman--Boucksom--Eyssidieux--Guedj--Zeriahi as ingredients in the  existence/properness principle of Rubinstein and the author, we show that existence of K\"ahler--Einstein metrics on log Fano pairs is equivalent to properness of the  K-energy in a suitable sense. To our knowledge, this result represents the first characterization of general log Fano pairs admitting K\"ahler--Einstein metrics. We also discuss the analogous result for K\"ahler--Ricci solitons on Fano varieties. 
\end{abstract}

\section{Introduction}

To describe the context of our results, let $(X,\o)$ be a compact K\"ahler manifold for the moment, and let $\mathcal H$ denote the space of K\"ahler metrics cohomologous to $\o$. A major direction of study in K\"ahler geometry is to find elements of $\mathcal H$ with special curvature properties. One can introduce Mabuchi's K-energy $\mathcal K:\mathcal H \to \Bbb R$, a functional whose critical points are exactly the constant scalar curvature K\"ahler (cscK) metrics of $\mathcal H$ \cite{mab}. For the precise definition of all functionals, as well as a detailed historical account we refer to \cite{dr2,r2}.

The K-energy is J-proper, if for any  $\{ \o_{u_j}\}_{j \in \Bbb N} \subset \mathcal H$ the following implication holds: 
\begin{equation}
\label{eq: Jpropeness_of_K}
J(\o_{u_j}) \to \infty \textup{ implies } \mathcal K(u_j) \to \infty,
\end{equation}
where $J: \mathcal H \to \Bbb R$ is Aubin's functional \cite{Aub}. Motivated by results of conformal geometry, in the 90's Tian conjectured that existence of constant scalar curvature K\"ahler  metrics in $\mathcal H$ should be equivalent to J-properness of $\mathcal K$ \cite[Remark 5.2]{t1},\cite{t3}, and this was proved for Fano manifolds with discrete automorphism group \cite{t2,tz}.  In \cite{pssw} the ``strong form" of the J-properness condition \eqref{eq: Jpropeness_of_K} was found, saying that $\mathcal K$ grows at least linearly with respect to $J$, and this stronger form has been later adopted throughout the literature.

When the automorpism group is non-discrete it was known that the conjecture cannot hold as stated above, and numerous modifications were proposed by Tian (see \cite[Conjecture 7.12]{t3}, \cite{t4}). In \cite{dr2}, Y.A. Rubinstein and the author disproved one of these conjectures and proved the rest for general Fano manifolds. Crucial to these developments was the existence/properness principle of \cite[Theorem 3.4]{dr2}, that placed Tian's predictions in the context of the $L^1$-Mabuchi geometry of $\mathcal H$, developed in \cite{da1,da2}. The resulting framework is very robust as one can obtain analogous theorems for both conical and soliton type K\"ahler--Einstein metrics \cite[Theorem 2.11,Theorem 2.12]{dr2}, and recently it was shown that in the general K\"ahler case, existence of a cscK metric in $\mathcal H$ implies J-properness of the K-energy \cite[Theorem 1.3]{bdl2}. 

In this work, we apply the existence/properness principle of \cite[Theorem 3.4]{dr2} in the context of singular varieties. To do this, first one has to extend the relevant $L^1$-Mabuchi metric geometry to normal K\"ahler spaces $(X,\o)$. This can be done in an economic manner, by noticing that the concrete formula for the $d_1$ metric from \cite{da2} makes sense in the singular setting as well, hence the extension of the Finsler geometric arguments of \cite{da2} can be avoided (Theorem \ref{thm: metrictheorem}). Finally, one uses this and the results of \cite{bbegz} as ingredients in the existence/properness principle to show that, on a  log Fano pair $(X,D)$, existence of a K\"ahler--Einstein metric is equivalent with J-properness of the K-energy (Theorem \ref{thm: propernesstheorem}). This  gives the first full characterization of general log Fano pairs admitting K\"ahler--Einstein metrics that we are aware of. 

We also mention applications to K-(poly)stability, in particular Theorem \ref{thm: propernesstheorem} and the techniques of \cite[Section 3]{bdl2} give an alternative analytic proof of a deep result of Berman from \cite{brm2} (Corollary \ref{cor: K-polystability}). All of our arguments work for the Ding functional as well, and we also give the analogous properness/existence theorem for K\"ahler--Ricci solitons on Fano varieties (Theorem \ref{thm: propernessthm_soliton}).

\section{Main results}

Suppose $Z$ is a compact complex space. A function $f: Z \to \Bbb R$ is smooth, if given any local embedding $Z \to \Bbb C^m$, $f$ is a restriction to $Z$ of a smooth function on $\Bbb C^m$. The corresponding space of smooth forms and vector-fields, holomorphic maps, plurisubharmonic (psh) functions etc. is defined similarly. 

Given a closed (1,1)-form $\alpha$ on $Z$, the corresponding space of smooth K\"ahler potentials  and $\alpha$-psh functions is denoted by 
$$\mathcal H_\alpha(Z) = \{ u \in C^\infty(Z) \textup{ s.t. } \alpha_u:=\alpha + i\ddbar u > 0\}.$$
$$\PSH_\alpha(Z) = \{ u \in L^{1}(Z)\textup{ is usc and } \ \alpha_u:= \alpha + i\ddbar u \geq 0\}.$$
The space ${\mathcal H}_\alpha(Z)$ is non-empty if and only if $[\alpha]_{dR}$ is a K\"ahler class. On the other hand, the space $\PSH_\alpha(Z)$ is non-empty if and only if $[\alpha]_{dR}$ is a  pseudo-effective class.

Following \cite{begz,bbegz}, the Monge--Amp\`ere (also Aubin--Yau, some times Aubin--Mabuchi) functional $I_\o: \PSH_\o(X) \cap L^\infty \to \Bbb R$ is the following map:
\begin{equation}\label{eq: AMdef}
I_{\o}(u) = \frac{1}{(n+1)V_\o} \sum_{j=0}^n\int_{X_{\textup{reg}}} u {\o}^j \wedge {\o}_u^{n-j},
\end{equation} 
where $V_\o = \int_Y \o^n$ (the notation $\mathcal E_\o(u)$ or $\textup{AM}_\o(u)$ is also often used in the literature).
As $I_{\o}$ is monotone, it makes sense to extend $I_\o$ to a functional $I_{\o} : \PSH_\o(X) \to [-\infty,\infty)$:
\begin{equation}\label{eq: AM_PSHext}
I_{\o}(v) = \inf_{u \in \textup{PSH}_\o(X) \cap L^\infty,  \ v \leq u} I_{\o}(u).
\end{equation}
According to the definitions of \cite{begz,bbegz}, $v \in \mathcal E^1_{\o}(X)$ if and only if $I_{\o}(v) > -\infty$. 

Let $S = \{ 0< \textup{Re }s <  1\} \subset \Bbb C$ and $u_0,u_1 \in \PSH_\o(X) \cap L^\infty$. We call the curve $[0,1] \ni t \to u_t \in \PSH_\o(X) \cap L^\infty$ the \emph{bounded geodesic} connecting $u_0,u_1$ if $v(s,x)=u_{\textup{Re }s}(x) \in \PSH_{\textup{pr}_2^* \o}(S \times X)$ and $v$ solves the following complex Monge-Amp\`ere equation on $S \times X_{\textup{reg}}$:
\begin{equation}\label{eq: weakgeodeq}
(\textup{pr}_2^* \o + i \ddbar v)^{n+1}=0.
\end{equation}
Following Berndtsson \cite{brn1} it was shown in \cite[Lemma 4.5]{bbegz} that any $u_0,u_1 \in \PSH_\o(X) \cap L^\infty$ can be connected by a bounded geodesic $t \to u_t$. 

Given more generally $u_0,u_1 \in \mathcal E^1_{\o}(X)$, let  $u^k_0=\max(u_0,-k),u^k_1=\max(u_1,-k) \in \PSH_\o(X)$. Let $[0,1] \ni t \to u^k_t \in \PSH_\o(X) \cap L^\infty$ be the bounded geodesics connecting $u^k_0,u^k_1$. We will show that the limit curve $[0,1] \ni t \to u_t = \lim_k u^k_t\in \mathcal E^1_{\o}(X)$ is well defined and  is called the \emph{finite energy geodesic} connecting $u_0,u_1 \in \mathcal E^1_{\o}(X)$. 

Given $u_0,u_1 \in \PSH_\o(X)$, one can introduce the following ``rooftop envelope": 
$$P_{\o}(u_0,u_1)=\sup \{v \in \PSH_\o(X) \ \textup{ s.t. } \ v \leq u_0,u_1\}.$$ 
By \cite[Theorem 3]{da1}, in case $(X,\o)$ is smooth K\"ahler, $P_{\o}(u_0,u_1) \in \mathcal E^1_{\o}(X)$ if and only if $u_0,u_1 \in \mathcal E^1_{\o}(X)$. We will show this in our singular setting as well (Theorem \ref{thm:rooftopenergystable}).

We can give now our first main result, extending some of the results of \cite{da1,da2} to the case of normal K\"ahler spaces with a Hodge K\"ahler structure:

\begin{theorem} \label{thm: metrictheorem} Let $(X,\o)$ be a projective normal K\"ahler space, with Hodge K\"ahler metric $\o$. Suppose $d_1^\o: \mathcal E^1_{\o}(X) \times \mathcal E^1_{\o}(X) \to \Bbb R$ is given by 
\begin{equation}\label{eq: d_1 formula}
d^\o_1(u_0,u_1):=I_{\o}(u_0) + I_{\o}(u_1)-2 I_{\o}(P_{\o}(u_0,u_1)).
\end{equation}
Then  $(\mathcal E^1_{\o}(X),d^\o_1)$ is a complete geodesic metric space, coinciding with the metric completion of $({\mathcal H}_{\o}(X),d_1^\o)$. Additionally, given $u_0,u_1 \in \mathcal E^1_{\o}(X)$, the finite energy geodesic $[0,1]\ni t \to u_t \in \mathcal E^1_{\o}(X)$ connecting $u_0,u_1$ is a $d_1^\o$-geodesic curve. 
\end{theorem}

Following the notation and terminology of \cite[Section 3.1]{bbegz}, by a \textit{pair} $(X,D)$ we mean a projective normal variety $X$ and an effective $\Bbb Q$-divisor $D$ such that $K_{X}+D$ is $\Bbb Q$-Cartier. Given a log-resolution $\pi: Y \to X$, we can always assume that $\pi$ is biholomoprhic on  $X_{reg}\setminus D$. For such a resolution there exists a $\Bbb Q$-divisor $\sum_j a_j E_j \subset Y$ such that 
\begin{equation}\label{eq: canbundid}
K_Y \cong \pi^* (K_{X}+D) + \sum_j a_j E_j. 
\end{equation}
The pair $(X,D)$ is klt (short for Kawamata log terminal) if $a_j >-1$. 
A \textit{log Fano pair} is a klt pair $(X,D)$ such that $-K_{X} -D$ is ample. Suppose $\o \in c_1(-K_{X}-D)$ is a fixed K\"ahler metric. We say that $u \in \PSH_{\o}(X)\cap L^\infty$ is a K\"ahler--Einstein potential for the pair $(X,D)$ if 
$u$ is smooth on $X_{reg}\setminus D$ and $\o_u$ satisfies the ``twisted" K\"ahler--Einstein equation on $X_{reg}$:
\begin{equation}\label{eq: KE_eq}
\Ric \o_u = \o_u + [D].
\end{equation}
By \cite[Theorem 4.8]{bbegz}, $u \in \PSH_{\o}(X)\cap L^\infty$ is a K\"ahler--Einstein potential if and only if $u \in \PSH_{\o}(X)\cap L^\infty$ is a minimizer of the (extended) K-energy functional $\mathcal K_D: \PSH_{\o}(X)\cap L^\infty \to [-\infty,\infty)$ or equivalently it is a minimizer the corresponding extended Ding functional $\mathcal F_D: \PSH_{\o}(X)\cap L^\infty \to \Bbb R$ (see Section \ref{sec: existence/properness thm} for precise definitions).

Recall that the functional $J: \mathcal E^1_\o(X) \to [0,\infty)$ is just defined by 
$$J(u)=\int_{X_{reg}} u \o^n - I_\o(u).$$ 
Let $G=\textup{Aut}_0(X,J,D)$ be the identity component of the group of holomorphic automorphisms of $(X,J)$ fixing $D$. The group $G$ acts on $\mathcal H_\o^0(X):={\mathcal H}_{\o}(X) \cap I_\o^{-1}(0)$, the \emph{space of K\"ahler metrics}, and as in \cite{dr2}, one can introduce $J_G: {\mathcal H}^0_\o(X)/G \to \Bbb R^+$, the ``$G$-dampened" version of the $J$ functional:
$$J_G(Gu)= \inf_{g \in G} J(g.u).$$
In case $G$ is trivial, $J_G$ just equals the regular $J$ functional. With Theorem \ref{thm: metrictheorem} as an ingredient in the existence/properness principle of Rubinstein and the author \cite[Theorem 3.4]{dr2}, we obtain the following result extending \cite[Theorem 2.12]{dr2} from the smooth K\"ahler case:
 
\begin{theorem}\label{thm: propernesstheorem} Suppose $(X,D)$ is log Fano pair and $\o \in c_1(-K_{X}-D)$ a K\"ahler metric. Then there exists a K\"ahler--Einstein potential $u \in \PSH_{\o}(X) \cap L^\infty$  if and only if $F \in \{\mathcal K_D,\mathcal F_D \}$ is $G$-invariant and for some $C,B> 0$ satisfies:
\begin{equation}\label{eq: main_prop_equation}
F(u) \geq CJ_{G}(Gu)-B, \ \ u \in {\mathcal H}^0_\o(X).
\end{equation}
\end{theorem} 
It is well known that existence of a K\"ahler--Einstein potential implies that $F$ is $G$-invariant. Also, \eqref{eq: main_prop_equation} automatically implies that $F \in \{\mathcal K_D,\mathcal F_D \}$ is $G$-invariant, hence this last condition is superfluous in the above theorem. However the $G$-invariance condition guarantees that $F$ descends to the quotient ${\mathcal H}_{0}(X)/G$, and we think that the estimate of \eqref{eq: main_prop_equation} should be naturally understood in this context (see also \cite{dr2}). 

In the context of singular varieties, the above result generalizes \cite[Theorem A(ii)]{bbegz}, which proves the if direction in the case of log Fano pairs with discrete automorphism group. In the singular context, to our knowledge, the only if direction has not been addressed in the literature so far, and Theorem \ref{thm: propernesstheorem} seems to give the first full characterization of log Fano pairs admitting K\"ahler--Einstein metrics. For a recent characterization of smoothable Fano variaties using K-polystability, see \cite{ssy}. 

In case X is smooth and $D= \sum_j (1-\beta_j)Y_j$ is a simple normal crossing divisor, the K\"ahler--Einstein metrics of the pair $(X;D)$ are in fact K\"ahler--Einstein edge metrics with cone angle $2\pi(1-\beta_j)$ along each $Y_j$ \cite{gp,jmr}  (\cite{jmr} treats only the case $j=1$, but gives more delicate regularity near $D$; see also \cite{mr}, \cite[Section 4.5]{r} for other approaches to the general case).

In case $(X,D)$ is of log general type, i.e., $K_X +D >0$, existence of K\"ahler--Einstein metrics has been discussed in detail in \cite{bg}. 

Given that the findings of \cite{da2, dr2} where used in establishing the results of \cite{bbj}, it would be interesting to see if Theorems \ref{thm: metrictheorem} and \ref{thm: propernesstheorem} could be useful in generalizing the main result of \cite{bbj} to log Fano pairs. Nonetheless, let us mention that as a corollary of Theorem \ref{thm: propernesstheorem} and the techniques of \cite[Section 3]{bdl2} we obtain an alternative proof of a deep result of Berman from \cite[Theorem 1.1]{brm2}:

\begin{corollary}\label{cor: K-polystability} Suppose $(X,\o)$ is a  Fano variety and $\o \in c_1(-K_X)$ is a K\"ahler metric. If there exists a K\"ahler--Einstein potential $u \in \textup{PSH}_\o(X) \cap L^\infty$, then $X$ is K-polystable. 
\end{corollary}

Finally, our results and methods carry over to the setting of K\"ahler--Ricci solitons on Fano varieties $(X,\o)$, explored in \cite{bwn}. We say that a normal K\"ahler space $(X,\o)$ is a \textit{Fano variety} if $(X,0)$ is a log Fano pair and $\o \in c_1(-K_X)$. Suppose $V$ is a holomorphic vector field on $X$, generating a $T$-torus action on $X$, and $\o$ is $T$-invariant. By $\mathcal H_\o^T(X),\textup{PSH}^T_\o(X),\mathcal E_\o^{1T}, \mathcal H_\o^{0T}$ we denote the appropriate set of $T$-invariant potentials and metrics. 

We say that $u \in \PSH^T_\o(X) \cap L^\infty$ is a K\"ahler--Ricci soliton potential if 
$u$ is smooth on $X_{reg}$, and on this set $\o_u$ satisfies
$$\Ric \o_u = \o_u + \mathcal L_V \o_u.$$
Let $G=\textup{Aut}_0(X,J,V)$ be the identity component of the group of holomorphic automorphisms commuting with $V$, and $F \in \{ \mathcal K_V,\mathcal F_V\}$ be the appropriate soliton version of the K-energy and Ding functional \cite{bwn}. We note the following result, extending both \cite[Theorem 2.11]{dr2} and \cite[Theorem 1.6]{bwn}:
\begin{theorem} \label{thm: propernessthm_soliton} Suppose $(X,\o)$ is a Fano variety and $V$ a holomorphic vector field on $X$, generating a $T$-torus action on $X$. Then a K\"ahler--Ricci soliton potential exists in $\PSH^T_\o(X) \cap L^\infty$ if and only if $F \in \{\mathcal K_V,\mathcal F_V \}$ is $G$-invariant and for some $C,B>0$ satisfies:
\begin{equation}\label{eq: main_prop_soliton_ineq}
F(u) \geq CJ_{G}(Gu)-B, \ \ u \in {\mathcal H}^{0T}_{\o}(X).
\end{equation}
\end{theorem}

\section{Preliminaries}

\subsection{Finite energy pluripotential theory on normal K\"ahler spaces}

We fix a projective compact normal K\"ahler space $(X,\o)$.  Let $\pi: Y \to X$ be a desingularization of $Y$ with $\eta = \pi^* \o$. In general $\eta$ is not K\"ahler, but $Y$ is projective, hence it admits a K\"ahler structure $(Y,\eta')$. 
As the fibers of $\pi$ are connected, and away from an analytic set $\pi$ is biholomorphic, it is standard to show that pulling back by $\pi$ induces the following identifications \cite[Theorem 1.7]{de1}:
\begin{equation}\label{eq: PSH_eqv}
\pi^*\PSH_{\o}(X) = \PSH_\eta(Y) \ \ \textup{ and } \ \  \pi^*(\PSH_{\o}(X) \cap L^\infty) = \PSH_\eta(Y) \cap L^\infty.\ \  
\end{equation}
Following \cite{egz1,bbegz}, using these identities one can translate questions about canonical metrics on the ample, but singular structure $(X,\o)$, to similar questions on the smooth, but degenerate structure $(Y,\eta)$, and we will resort to this trick  throughout this work. 

The fact that  the inclusion map embeds $\PSH_\eta(Y)$ into $\PSH_{\eta+ \varepsilon \eta'}(Y)$ for any $\varepsilon \geq 0$, will allow us to ``approximate" the degenerate structure $(Y,\eta)$ with the K\"ahler structures $(Y,\eta+\varepsilon\eta')$.

In analogy with \eqref{eq: AMdef}, the Monge-Amp\`ere energy $I_\eta: \PSH_\eta(Y) \cap L^\infty \to \Bbb R$ of $(Y,\eta)$  is just the following map:
\begin{equation}\label{eq: AMdef_semipos}
I_{\eta}(u) = \frac{1}{(n+1)V_\eta} \sum_{j=0}^n\int_{Y} u {\eta}^j \wedge {\eta}_u^{n-j},
\end{equation}
where $V_\eta = \int_Y \eta^n$. As $I_{\eta}$ is monotone, it makes sense to extend this definition to a functional $I_{\eta} : \PSH_\eta(Y) \to [-\infty,\infty)$ the same way as in \eqref{eq: AM_PSHext}. As $\eta_u$ does not charge analytic sets of $Y$ (or more generally pluripolar sets), we get that 
$$I_\eta(\pi^* u)=I_\o(u), \ \  u \in \PSH_\o(X).$$
By definition, $v \in \mathcal E^1_{\eta}(Y)$ if and only if $I_{\eta}(Y) > -\infty$ and we have the following identity between $\mathcal E^1_\o(X)$ and $\mathcal E^1_\eta(Y)$:
$$\pi^*\mathcal E^1_\o(X)=\mathcal E^1_\eta(Y).$$
A similar phenomenon to the above also holds for envelopes and bounded/finite geodesics as well, as detailed later in this paper.

\subsection{Envelopes in finite energy classes}

In this subsection $(X,\o)$ is a normal K\"ahler space with a Hodge metric, i.e., $[\o]$ is integral in   $H^2(X,\Bbb R)$, and we fix a desingularization $\pi: Y \to X$ with $\eta: = \pi^* \o$.
Let us first note the following approximation result, which will be crucial for developments in this paper:
\begin{theorem}\textup{(\cite[Corollary C]{cgz})} \label{thm: EGZapprox} For any $u \in \PSH_\eta(Y)$, there exists  a decreasing sequence $\{u_j\}_j \subset \pi^* {\mathcal H}_{\o}(X) \subset \PSH_\eta(Y) \cap C^\infty$ such that $u_j \searrow u$ pointwise. 
\end{theorem}

Recall from the previous subsection that formula \eqref{eq: AMdef} gives an extension of $I_\eta$ to $\PSH_\eta(Y)$. The following result additionally establishes the continuity of $I_\eta$ along monotonic sequences:
\begin{proposition}\textup{(\cite[Theorem 2.17]{begz})} \label{prop: AM_mon_cont} Suppose $u_j,u \in \PSH_\eta(Y)$. If $u_j$ increases a.e. or decreases to $u$, then $\lim_jI_\eta(u_j) = I_\eta(u)$. 
\end{proposition}

Given $u \in \PSH_\eta(Y)$, by \cite{begz} one can introduce the non-pluripolar complex Monge-Ampere measure $\eta_u^n$ as a limit of an increasing sequence of measures:
$$\eta_u^n = \lim_{j \to \infty} \mathbbm{1}_{\{u> -j\}} \eta_{\max(u,-j)}^n,$$
and the inequality $\int_Y \eta_u^n \leq \int_Y \eta^n$ holds for the total measure. As we recall now, the are multiple ways to characterize membership in $\mathcal E^1_\eta(Y)$:
\begin{theorem}\textup{\cite[Section 2.2]{begz}} \label{thm: E1eqvdef} Suppose $u \in \PSH_\eta(Y)$. Then the following are equivalent:\\
\noindent(i) $u \in \mathcal E^1_\eta(Y),$ i.e., $I_\eta(u) > -\infty.$\\
\noindent (ii) $\int_Y \eta_u^n = \int_Y \eta^n$ and $\int_Y |u| \eta_u^n < \infty$. \\
\noindent (iii) For all $u_j \in \PSH_\eta(Y) \cap L^\infty$ decreasing to $u$ we have $\limsup_j \int_Y |u_j| \eta_{u_j}^n < \infty$.\\
(iv) For all $u_j \in \PSH_\eta(Y) \cap L^\infty$ decreasing to $u$ we have $\lim_j I_\eta(u_j) > -\infty$.\\
\noindent (v) For some $u_j \in \PSH_\eta(Y) \cap L^\infty$ decreasing to $u$ we have $\limsup_j \int_Y |u_j| \eta_{u_j}^n < \infty$.\\
\noindent (vi) For some $u_j \in \PSH_\eta(Y) \cap L^\infty$ decreasing to $u$ we have $\lim_j I_\eta(u_j) > -\infty$.\\
\end{theorem}
Let us recall the following crucial regularity result of Berman--Demailly that applies in our general setting. As we are dealing with integral K\"ahler classes, we could also use the results of \cite{brm3}. Note that   by a theorem of Demailly--P\u aun \cite[Theorem 0.5]{dp} as $\int_Y \eta^n >0$, the class $[\eta]$ is big. 
\begin{theorem}\textup{(\cite[Theorem 1.4]{bd},\cite[Theorem 1.1]{brm3})} For any $h \in C^\infty(Y)$ we have that $\Delta^\eta P_\eta(h)$ is locally bounded in $Y\setminus \pi^{-1}(X_{sing})$, where
$$P_\eta(h)= \sup\{v \in \PSH_\eta(Y), \ \textup{s.t. } v \leq h \}.$$
\end{theorem}

In its original form this theorem contains more precise information about $\Delta^\eta P_\eta(h)$, but for the purposes of proving the result below, the above statement will suffice.  We now generalize the ideas of \cite[Section 2.2]{da1}:

\begin{proposition}\label{prop: MA_form} Given $u_0,u_1 \in \mathcal \pi^*\mathcal H_\o(X)$ such that $\{u_0-u_1 =0 \} \subset Y$ contains no critical points of $u_1 - u_0$. We introduce $\Lambda_{u_0} = \{ P_\eta(u_0,u_1)=u_0\}$ and  $\Lambda_{u_1} = \{ P_\eta(u_0,u_1)=u_1\}$. Then the following partition formula holds for the complex Monge--Amp\`ere measure of $P_\eta(u_0,u_1)$:
\begin{equation}\label{MA_partition}
\eta_{P_\eta(u_0,u_1)}^n= \mathbbm{1}_{\Lambda_{u_0}}\eta_{u_0}^n + \mathbbm{1}_{\Lambda_{u_1}}\eta_{u_1}^n.
\end{equation}
\end{proposition}
\begin{proof}
As a consequence of \cite[Proposition 4.5]{dr1} there exists an open neighborhood $U$ containing the smooth hypersurface $\{u_0 - u_1=0 \}$, such that $U \subset \{ P_\eta(u_0,u_1) < \min\{ u_0,u_1\}\}$. Consequently, one can construct $v \in C^\infty(Y)$ such that $\min\{u_0,u_1 \} \geq v \geq P_\eta(u_0,u_1)$ and $v = \min \{v_0,v_1 \}$ on $Y\setminus U$, yielding that $P_\eta(v_0,v_1)=P_\eta(v)$. Hence we can apply the above theorem to conclude that $\Delta^\eta P_\eta(u_0,u_1)=\Delta^\eta P_\eta(v)$ is locally bounded in $Y \setminus \pi^{-1}(X_{sing})$. 

From \cite[Corollary 9.2]{bt} it then follows that $\eta_{P_\eta(u_0,u_1)}^n$ is concentrated on the coincidence set $\Lambda_{u_0} \cup \Lambda_{u_1}.$ Having bounded Laplacian on $Y \setminus \pi^{-1}(X_{sing})$ implies that all second order partials of $P_\eta(u_0,u_1)$ are in any $L^p_{loc}(Y\setminus \pi^{-1}(X_{sing})), \ p <\infty$. It follows from \cite[Chapter 7, Lemma 7.7]{gt} that on $\Lambda_{u_0} \setminus \pi^{-1}(X_{sing})$ all the second order partials of $P_\eta(u_0,u_1)$ and $u_0$ agree, and an analogous statement holds on $\Lambda_{u_1} \setminus \pi^{-1}(X_{sing})$. 
As $\pi^{-1}(X_{sing}) \subset Y$ is an analytic set, we have $\int_{\pi^{-1}(X_{sing})} \eta^n_{P_\eta(v_0,v_1)}=0$. Putting everything together, using \cite[Proposition 2.1.6]{bl1}  one can write:
\begin{flalign*}
\eta_{P_\eta(u_0,u_1)}^n&=\mathbbm{1}_{\Lambda_{u_0} \cup \Lambda_{u_1}}\eta_{P_\eta(u_0,u_1)}^n =\mathbbm{1}_{\Lambda_{u_0}}\eta_{u_0}^n + \mathbbm{1}_{\Lambda_{u_1} \setminus \Lambda_{u_0}}\eta_{u_1}^n.
\end{flalign*}
As we argued in the beginning of the proof, $\{u_1 - u_0=0 \} \subset \{P_\eta(u_0,u_1) < \min \{ u_0,u_1\} \} $ hence ${\Lambda_{u_1} \setminus \Lambda_{u_0}}=\Lambda_{u_1}$, finishing the proof.
\end{proof}

The following theorem now partially generalizes \cite[Theorem 2]{da1} from the smooth K\"ahler case:

\begin{theorem} \label{thm:rooftopenergystable}Suppose $u_0,u_1 \in \mathcal E^1_\eta(Y)$. Then $P_\eta(u_0,u_1) \in \mathcal E^1_\eta(Y)$. 
\end{theorem}
\begin{proof} We can assume without loss of generality that $u_0,u_1 < -1$. Using Theorem \ref{thm: EGZapprox}, suppose $\phi_j,\psi_j \in \pi^*\mathcal H_\o(X)$ are such that $0> \phi_j \searrow u_0$ and $0> \psi_j \searrow u_1$ pointwise. Such sequences exist by Theorem \ref{thm: EGZapprox}. By adding a small constant if necessary, we can also arrange that $\{\psi_j -\phi_j=0 \} \subset X$ contains no critical point of $\psi_j -\phi_j$. Using the last result we can start to write:
\begin{flalign}\int_Y |P_\eta(\phi_j,\psi_j)|\eta_{P_\eta(\phi_j,\psi_j)}^n \nonumber &= \int_{\{P_\eta(\phi_j,\psi_j) = \phi_j\}}|\phi_j|\eta_{\phi_j}^n + \int_{\{P_\eta(\phi_j,\psi_j) = \psi_j\}}|\psi_j|\eta_{\psi_j}^n \nonumber \\
&\leq \int_Y|\phi_j|\eta_{\phi_j}^n + \int_Y|\psi_j|\eta_{\psi_j}^n \nonumber.
\end{flalign}
As we have $P_\eta(\phi_j,\psi_j) \searrow P_\eta(u_0,u_1)$ and $u_0,u_1 \in \mathcal E^1_\eta(Y)$, letting $j \to \infty$ in the above estimate, by Theorem \ref{thm: E1eqvdef} we obtain that $P_\eta(u_0,u_1) \in \mathcal E^1_\eta(Y)$, finishing the proof.
\end{proof}

\subsection{The metric space $(\mathcal H_\o(X),d_1)$ for K\"ahler manifolds $(X,\o)$}

For this subsection, assume that $(X,\o)$ is a smooth K\"ahler manifold. We now explain some of the results of \cite{da2}, that give Theorem \ref{thm: metrictheorem} in the K\"ahler case. In \cite{da2}, the author introduced an $L^1$-type Finsler metric on the Fr\'echet manifold $\mathcal H_\o(X)$:
\begin{equation}
\| \psi\| = \int_X |\psi| \o_u^n, \ \ \psi \in T_u \mathcal H_\o.\label{eq: Finsler}
\end{equation}
The $L^2$ analog of this metric is the much studied Riemannian metric of Mabuchi--Semmes--Donaldson \cite{mab,se,do1}. Given this Finsler metric one can compute the length $l(\gamma_t)$ of smooth curves $t\to \gamma_t$ of $\mathcal H_\o(X)$, ultimately yielding the path length metric:
$$d_1^\omega(u_0,u_1)= \inf\{l(\gamma_t)| \ t \to v_t \textup{ is smooth and connects }u_0,u_1\}.$$

Then in \cite[Corollary 4.14]{da2} we showed that $d_1^\o$ thus defined satisfies \eqref{eq: d_1 formula}.
This approach is certainly more attractive from the geometric point of view and generalizes to other Orlicz-Finsler type structures, as detailed in \cite{da1,da2}. However, most of these Finsler geometric arguments do not carry over directly to our degenerate/singular setting, not to mention the delicate estimates about weak geodesics from \cite{c1,bl1}, which are one of the building blocks of the theory. 

To circumvent these difficulties, in the degenerate/singular case we will ultimately define the metric $d_1^\o$ using \eqref{eq: d_1 formula}, and will use approximation by K\"ahler classes to establish many additional properties of $(\mathcal H_\o(X),d_1)$. Before we do this, let us collect the results of \cite{da2} that are most relevant for this paper:

\begin{theorem}[\cite{da2}] \label{thm: metrictheorem_Kahler} Let $(X,\o)$ be a smooth compact K\"ahler manifold. Suppose $d_1$ is the path length metric of the Finsler structure \eqref{eq: Finsler}. Then this metric extends to $d_1: \mathcal E^1_{\o}(X) \times \mathcal E^1_{\o}(X) \to \Bbb R$ and satisfies the identity
\begin{equation*}
d^\o_1(u_0,u_1)=I_{\o}(u_0) + I_{\o}(u_1)-2 I_{\o}(P_{\o}(u_0,u_1)).
\end{equation*}
The resulting structure $(\mathcal E^1_{\o}(X),d_1^\o)$ is a complete geodesic metric space, coinciding with the metric completion of $({\mathcal H}_{\o}(X),d_1^\o)$. Additionally, given $u_0,u_1 \in \mathcal E^1_{\o}(X)$, the finite energy geodesic $[0,1]\ni t \to u_t \in \mathcal E^1_{\o}(X)$ is a $d_1$-geodesic curve.
\end{theorem}
 
In the following proposition let us recall a number of consequences/ingredients of this theorem that will be later generalized to the non-K\"ahler case:

\begin{proposition}\label{thm: da2properties} Suppose $u_0,u_1 \in \mathcal E^1_\o(X)$. The following hold:\\
\noindent(i)\textup{\cite[Proposition 4.12]{da2}} For $v \in \mathcal E^1_\o(X)$ we have $d^\o_1(P_\o(v,u_0),P_\o(v,u_1)) \leq d_1^\o(u_0,u_1).$\\
\noindent(ii)\textup{\cite[Lemma 4.15]{da2}} $I_\o$ is $d_1$-Lipschitz continuous, and for any finite energy geodesic $t \to u_t$, the map $t \to I_\o(u_t)$ is linear.
\end{proposition}

\subsection{The existence/properness principle of \cite{dr2}}

We recall here the general framework of \cite{dr2} suitable for proving equivalence between existence of canonical metrics in K\"ahler geometry and properness of certain energy functionals.
The data $(\mathcal R,d,F,G)$ is defined by the following ``axioms":
\begin{enumerate}[label = (A\arabic*)]
\setlength{\itemsep}{1pt}
    \setlength{\parskip}{1pt}
    \setlength{\parsep}{1pt} 
  \item\label{a1} $(\mathcal R,d)$ is a metric space with a 
distinguished 
element $0\in\mathcal R$,
        whose metric completion is denoted $(\overline{\mathcal R},d)$.
  \item\label{a2} $F : \mathcal R \to \Bbb R$ 
is $d$-lsc. Let  
        $F: \overline{\mathcal R} \to \Bbb R \cup \{ +\infty\}$ 
        be the largest lsc extension
        of $F$:
$$F(u) = \sup_{\varepsilon > 0 } \bigg(\inf_{\substack{v \in \mathcal R\\ d(u,v) \leq \varepsilon}} F(v) \bigg), \ \ u \in \overline{ \mathcal R}.$$
\item\label{a3} 
The set of minimizers of $F$ on $\overline{\mathcal R}$ is denoted by
$$
\mathcal M:= 
\Big\{ u \in \overline{\mathcal R} \ : \ F (u)= 
\inf_{v \in \overline{\mathcal R}} F(v)
\Big\}.
$$
\item\label{a4} $G$ is a group  
  acting on ${\mathcal R}$ by 
  $G\times{\mathcal R}\ni(g,u) \to g.u\in {\mathcal R}$. Denote by
${\mathcal R} /G$ the orbit space, by $Gu\in{\mathcal R} /G$ the 
  orbit of $u\in{\mathcal R}$, 
  and~define~$d_G:{\mathcal R}/G\times {\mathcal R}/G\ra\RR_+$ by
$$
d_G(Gu,Gv):=\inf_{f,g\in G}d(f.u,g.v).
$$
\end{enumerate}
The data $(\mathcal R,d,F,G)$
satisfies the following ``properties":
\begin{enumerate}[label = (P\arabic*)]
\setlength{\itemsep}{1pt}
    \setlength{\parskip}{1pt}
    \setlength{\parsep}{1pt} 
  \item\label{p1}
 For any $u_0,u_1 \in \mathcal R$ there exists a $d$--geodesic segment $[0,1] \ni t \mapsto u_t \in \overline{\mathcal R}$ for which
 $t\mapsto F(u_t) \textup{ is continuous and convex on }[0,1].$
  \item \label{p2} 
If  $\{u_j\}_j\subset \overline{\mathcal R}$  satisfies 
$\lim_{j\ra\infty}F(u_j)= \inf_{\overline{\mathcal R}} F$, and
for some $C>0$, $d(0,u_j) \leq C$ for all $j$, then there exists a $u\in \mathcal M$ and a subsequence $\{u_{j_k}\}_k$ 
$d$-converging to $u$.
  \item\label{p3} 
  $\mathcal M \subset \mathcal R.$
  \item\label{p4} $G$ acts on ${\mathcal R}$ by $d$-isometries.
  \item \label{p5} $G$ acts on $\mathcal M$ transitively.
  \item \label{p6} If $\mathcal M \neq \emptyset$,  
then for any $u,v \in \mathcal R$ there exists $g \in G$
such that $d_G(Gu,Gv)=d(u,g.v)$.

  \item \label{p7} For all $u,v \in \mathcal R$ and $g \in G$,
$F(u) -F(v)=F(g.u) - F(g.v)$.
\end{enumerate}
The following result provides the framework that relates existence of canonical \K metrics
to properness of functionals with respect to a certain the Finsler metric:

\begin{theorem}\textup{\cite[Theorem 3.4]{dr2}}
\label{thm: ExistencePrinc}
Let $(\mathcal R,d,F,G)$ be as above, satisfying \ref{a1}-\ref{a4} and \ref{p1}-\ref{p7}. The following conditions are equivalent:\\
\noindent
(i) $\mathcal M$ is nonempty.

\noindent
(ii) $F:{\mathcal R}\ra \RR$ 
is $G$-invariant, and
for some $C,D>0$,
\begin{equation}\label{Dproperness}
    F(u) \geq Cd_G(G0,Gu)-D,\ \textup{ for all } u \in \mathcal R.
\end{equation}
\end{theorem}

\section{The metric space $(\mathcal H_\o(X),d_1)$ for normal K\"ahler spaces $(X,\o)$}

For this section assume again that $(X,\o)$ is a normal K\"ahler space with Hodge metric $\o$, and $\pi: Y \to X$ is a desingularization with $\eta = \pi^* \o$. For a while, we will work with the smooth but semi-positive structure $(Y,\eta)$ and later we will relate everything to $(X,\o)$. 

In analogy with the introduction, the curve $[0,1] \ni t \to u_t \in \PSH_\eta(Y) \cap L^\infty$ is a bounded  geodesic connecting $u_0,u_1  \in \PSH_\eta(Y) \cap L^\infty$ if its complexification $u$ solves the following boundary value problem on $S \times Y$:
\begin{equation}\label{eq: weakgeodeq1}
(\textup{pr}_2^* \eta + i \ddbar u)^{n+1}=0, \ \ \ u(0,y)=u_0(y), \ \ \ u(1,y)=u_1(y).
\end{equation}
By an observation of Berndtsson in the K\"ahler case \cite[Section 2.1]{brn1} this boundary value problem is always solvable, with uniformly $t$-Lipschitz solution, and his arguments go trough without modifications in the semi-positive case as well \cite[Lemma 4.5]{bbegz}. More precisely, $u$ arises as the following upper envelope:
\begin{equation}\label{eq: boundedgeodformulaonY}
u = \sup \{v \in \PSH_{\textup{pr}_2^* \eta}(S \times Y) \textup{ s.t. } v(t+ir,y)=v(t,y) \textup{ and } v_0 \le u_0, v_1 \leq u_1 \}.
\end{equation}
We introduce $d_1^\eta$ in the obvious manner: 
$$d_1^{\eta}(u_0,u_1) = I_{\eta}(u_0)+ I_{\eta}(u_1)-2I_{\eta}(P_{\eta'}(u_0,u_1)), \ \ u_0,u_1 \in \PSH_{\eta}(Y)\cap L^\infty.$$ 
As we will see soon, with this definition $d_1^{\eta}$ is indeed a metric on $\PSH_{\eta}(Y)\cap L^\infty$, whose completion will be $\mathcal E^1_\eta(Y)$. The following approximation result rests on the trivial observation that $\PSH_{\eta}(Y)\cap L^\infty \subset \PSH_{\eta+\varepsilon \eta'}(Y)\cap L^\infty$. By \cite[Example 3.5]{dn}, the analogous inclusion does not hold for $\mathcal E_{\eta}^1(Y)$ and $ \mathcal E^1_{\eta+\varepsilon \eta'}(Y)$, and because of this, we first have to focus on bounded potentials only:

\begin{proposition} \label{prop: lotsoflimits} Suppose $u_0,u_1 \in \PSH_\eta(Y) \cap L^\infty$ and  $\varepsilon_j \to 0$. Assume $u^j_0,u^j_1 \in \PSH_{\eta + \varepsilon_j \eta'}(Y) \cap L^\infty$ satisfies $u^j_0\searrow u_0$ and $u^j_1\searrow u_1 $ as $j \to \infty$. Additionally, let $[0,1] \ni t \to u^{\varepsilon_j}_t,u_t \in \PSH_{\eta + \varepsilon_j \eta'}(Y) \cap L^\infty$ be the bounded geodesics connecting $u_0^j,u_1^j$ and $u_0,u_1$, with respect to the different K\"ahler classes. The following hold:\\
(i) $P_{\eta + \varepsilon_j \eta'}(u_0^j,u_1^j) \searrow P_{\eta}(u_0,u_1)$ pointwise as $j \to \infty$.\\
(ii) $u_t^{\varepsilon_j} \searrow u_t$ pointwise as $j \to \infty$.\\
(iii)  $d_1^\eta(u_0,u_1)= \lim_{\varepsilon \to 0}d_1^{\eta + \varepsilon_j\eta'}(u_0^j,u_1^j).$
\end{proposition}
\begin{proof} As $\PSH_{\eta+ \varepsilon'\eta'}(Y) \subset \PSH_{\eta + \varepsilon \eta'}(Y)$ for $0 \leq \varepsilon' < \varepsilon$, it follows that $P_{\eta + \varepsilon_j \eta'}(u_0^j,u_1^j) \leq \min(u_0^j,u_1^j)$ is decreasing in $j$. Consequentially, the limit $h:=\lim_{j \to \infty}P_{\eta + \varepsilon_j \eta'}(u_0^j,u_1^j)$ is well defined and satisfies $h\in \PSH_\eta(Y), \ h \leq u_0,u_1$, implying (i).

Note that $u^{\varepsilon_j}(s,y)= u^{\varepsilon_j}_{\textup{Re }s}(y) \in \PSH_{\textup{pr}_2^*\eta + \varepsilon_j \textup{pr}_2^*\eta'}(S \times Y)$ is uniformly $t$-Lipschitz continuous independently of $j$ (see the arguments in \cite[Section 2.1]{brn1}) and solves the equation 
$$(\textup{pr}_2^* (\eta + \varepsilon_j \eta') + i \ddbar u^{\varepsilon_j})^{n+1}=0.$$ 
Clearly, $u^{\varepsilon_j} \in \PSH_{\textup{pr}_2^*\eta + \varepsilon\textup{pr}_2^*\eta'}(S \times Y)$ for any $\varepsilon$ satisfying $\varepsilon_j < \varepsilon $. By \eqref{eq: boundedgeodformulaonY}, it readily follows that $u^{\varepsilon_j}$ is decreasing in $j$. According to Bedford--Taylor theory \cite{bt}, the decreasing limit curve $t \to v_t= \lim_{j \to \infty} u^{\varepsilon_j}_t \in \PSH_{\eta}(Y)$ satisfies 
$$(\textup{pr}_2^* \eta + i \ddbar v)^{n+1}=0,$$
and we have $v_0=u_0$ and $v_1 = u_1$. Applying the comparison principle \cite[Theorem 21]{bl1} (whose proof from the K\"ahler case extends to our more general situation) for the $t$-Lispchitz potentials $v,u \in \PSH_
{\textup{pr}_2^*\eta}(S \times Y)$   we obtain that $v = u$, proving (ii).

To conclude (iii) we have to argue that 
\begin{flalign*}
I_\eta(u_0) + I_\eta(u_1) & - 2 I_\eta(P_\eta(u_0,u_1))=\\ =&\lim_{j \to \infty}\big(I_{\eta+ \varepsilon_j \eta'}(u_0^j) + I_{\eta+ \varepsilon_j \eta'}(u_1^j) - 2 I_{\eta+ \varepsilon_j \eta'}(P_{\eta+ \varepsilon_j \eta'}(u_0^j,u_1^j))\big),
\end{flalign*}
but this follows from classical Bedford--Taylor theory \cite{bt}, given the defining expression for $I_\eta$ \eqref{eq: AMdef_semipos} and the fact that $u^j_0 \searrow u_0$, $u^j_1 \searrow u_1$, $P_{\eta+\varepsilon_j \eta'}(u^j_0,u^j_1) \searrow P_{\eta}(u_0,u_1)$.
\end{proof}

Using the last statement of the previous proposition, it follows that $d^1_\eta$ is a pseudometric on $\PSH_\eta(Y) \cap L^\infty$. Now we argue that $d_1$ is actually a bona-fide metric, and we also look at its geodesics:

\begin{proposition} $(\PSH_\eta(Y) \cap L^\infty,d_1^\eta)$ is a geodesic metric space. Given $u_0,u_1 \in \PSH_\eta(Y) \cap L^\infty$, the weak geodesic $[0,1] \ni t \to u_t \in \PSH_\eta(Y) \cap L^\infty$ joining $u_0,u_1$ is a $d_1$-geodesic.
\end{proposition}
\begin{proof} According to the previous theorem $d_1^\eta(u_0,u_1)=\lim_{\varepsilon \to 0}d_1^{\eta+\varepsilon \eta'}(u_0,u_1)$, hence the triangle inequality holds for $d^\eta_1$.

If $d_1^\eta(u_0,u_1)=0$ then by the monotonicity of $I_\eta$ we obtain $I_\eta(u_0)=I_\eta(u_1)=I_\eta(P_\eta(u_0,u_1))$. Using the next proposition, this implies $u_0=P_\eta(u_0,u_1)=u_1$, hence $d_1^\eta$ is non-degenerate.

Suppose $[0,1] \ni t \to u_t^\varepsilon \in \PSH_{\eta + \varepsilon \eta'}(X)\cap L^\infty$ is the weak geodesic joining $u_0,u_1 \in \PSH_{\eta + \varepsilon \eta'}(X)\cap L^\infty$. Then by Theorem \ref{thm: metrictheorem_Kahler} we have that $t \to d_1^{\eta + \varepsilon \eta'}(u_0,u^\varepsilon_t), \ t \in [0,1]$ is linear. By Proposition \ref{prop: lotsoflimits}(iii) it follows that $t \to \lim_{\varepsilon \to 0}d_1^{\eta + \varepsilon \eta'}(u_0,u^\varepsilon_t)=d^\eta_1(u_0,u_t)$ is linear as well. One proves similarly that $t \to d_1^\eta(u_l,u_t)$ is linear for any $l \in [0,1]$, implying that $t \to u_t$ is indeed a $d_1^\eta$-geodesic.
\end{proof}

The following cocycle formula of the Aubin--Mabuchi energy is classical in the K\"ahler case and can be extended to $\mathcal E^1_\eta(X)$ without changes (see \cite{begz, bbegz}):
\begin{equation}\label{eq: AM_diff}
I_\eta(u)-I_\eta(v)= \frac{1}{(n+1) V} \sum_{j=0}^{n}\int_Y (u-v) \eta_u^{j} \wedge \eta_v^{n-j}, \ \  u,v \in \mathcal E^1_\eta(Y).\end{equation}
The following result is a kind of ``domination principle" for the Aubin--Mabuchi energy. It is a generalization of the analogous result from the K\"ahler case \cite[Proposition 4.2]{bdl1}:

\begin{proposition} \label{prop: energy_dom}Suppose $u,v \in \mathcal E^1_\eta(Y)$ with $u \geq v$. If $I_\eta(u)=I_\eta(v)$, then $u=v$.
\end{proposition}

\begin{proof}
Since $u \geq v$ it follows from \eqref{eq: AM_diff} that
\begin{equation}\label{egy}
\int_Y (u - v)\eta_u^j\wedge \eta_v^{n-j}=0, \ j = 0,...,n.
\end{equation}
As $\eta$ is degenerate only on a pluripolar set, we are finished if we can prove that
\begin{equation}\label{main_id}
\int_Y i\partial(u - v)\wedge \bar\partial(u - v)\wedge \eta^{n-1}=0.
\end{equation}
The first step is to prove that
\begin{equation}\label{main_id1}
\int_Y i \partial(u - v)\wedge\bar\partial(u - v)\wedge \eta \wedge \eta_u^j\wedge \eta_v^{n-j-2}=0, \ j = 0,...,n-2.
\end{equation}
It follows from (\ref{egy}) that
\begin{flalign}\label{CauchySchwarz}
0=&\int_Y (u - v)i\partial\bar\partial(u - v)\wedge \eta_u^j\wedge \eta_v^{n-j-1} \nonumber\\
 =&-\int_Y i \partial(u - v)\wedge \bar\partial(u - v)\wedge \eta_u^j\wedge \eta_v^{n-j-1}, \ j = 0,...,n-1.
\end{flalign}
Using this we now write
\begin{flalign*}
&\int_Y i \partial(u - v)\wedge \bar\partial(u - v)\wedge\eta\wedge\eta_u^j\wedge \eta_v^{n-j-2}=\\
&=-\int_Y i \partial(u - v)\wedge \bar\partial(u - v)\wedge i\partial \bar \partial v\wedge\eta_u^j\wedge \eta_v^{n-j-2}\\
&=\int_Y i \partial(u - v)\wedge i\partial\bar\partial(u - v)\wedge \bar \partial v\wedge\eta_u^j\wedge \eta_v^{n-j-2}\\
&=\int_Y i \partial(u - v)\wedge \bar \partial v\wedge \eta_u\wedge\eta_u^j\wedge \eta_v^{n-j-2}-\\
&-\int_Y i \partial(u - v)\wedge \bar \partial v\wedge \eta_v\wedge\eta_u^j\wedge \eta_v^{n-j-2}, \ j = 0,...,n-2.
\end{flalign*}
Using the Cauchy-Schwarz inequality it readily follows from \eqref{CauchySchwarz} that both of the terms in this last difference are zero, proving (\ref{main_id1}). Continuing this inductive process we arrive at (\ref{main_id}).
\end{proof}

It is possible to extend the metric structure from $(\PSH_\eta(Y) \cap L^\infty,d_1^\eta)$ to $(\mathcal E^1_\eta(Y),d^\eta_1)$ and we do this now. Indeed, given $u_0, u_1 \in \mathcal E^1_\eta(Y)$, one can introduce:
\begin{equation}\label{eq: d_1sequence_ext}
d_1^\eta(u_0,u_1):= \lim_{j \to \infty}d_1^\eta(\max(u_0,-j),\max(u_1,-j)).
\end{equation}
A few comments are in order about why this limit exsits. By definition of $\mathcal E^1_\eta(Y)$, the Aubin--Mabuchi energy extends with finite values to $\mathcal E^1_\eta(Y)$. 
By Theorem \ref{thm:rooftopenergystable} the operation $(u_0,u_1) \to P_\eta(u_0,u_1)$ is closed inside $\mathcal E^1_\eta(Y)$ and also $P_\eta (\max(u_0,-j),\max(u_1,-j)) \searrow P_\eta(u_0,u_1)$. Putting everything together, Proposition \ref{prop: AM_mon_cont} implies that the limit in \eqref{eq: d_1sequence_ext}  exists and equals the following explicit expression:
\begin{equation}\label{eq: d_1explicit_ext}
d_1^\eta(u_0,u_1)= I_\eta(u_0) + I_\eta(u_1)-2I_\eta(P_\eta(u_0,u_1)), \ u_0,u_1 \in \mathcal E^1_\eta(Y).
\end{equation}
Before we proceed, let us recall that the finite energy geodesics connecting $u_0,u_1 \in \mathcal E^1_\eta(X)$ is just the decreasing limit of  the bounded geodesics $t \to u^j_t$ connecting $u^j_0 := \max(u_0,-j), \ u^j_0 := \max(u_0,-j)$. It is straightforward to verify that for $u$ thus defined, the upper envelope formula \eqref{eq: boundedgeodformulaonY} also holds. As it will be clear in a moment, this curve actually stays inside $\mathcal E^1_\eta(Y)$:

\begin{proposition}\label{prop: E1extension} With the explicit extension \eqref{eq: d_1explicit_ext} the structure $(\mathcal E^1_\eta(Y),d_1^\eta)$ is a geodesic metric space. Given $u_0,u_1 \in \mathcal E^1_\eta(Y)$, the finite energy geodesic $[0,1] \ni t \to u_t \in \mathcal E^1_\eta(Y)$ joining $u_0,u_1$ is a $d_1^\eta$-geodesic.
\end{proposition}
\begin{proof}The triangle inequality for $d_1^\eta$ follows from \eqref{eq: d_1sequence_ext}. The non-degeneracy of $d_1^\eta$ follows again from Proposition \ref{prop: energy_dom}. 

Suppose $t \to u^j_t$ is the bounded geodesic connecting $\max(u_0,-j)$ and $\max(u_1,-j)$ for all $j$. As $u^j_t \searrow u_t, \ t \in [0,1]$ and $t \to I_\eta(u^j_t)$ is linear, by Proposition \ref{prop: AM_mon_cont} it follows that $t \to I_\eta(u_t)$ is linear or maybe identically equal to $-\infty$. Because $u_0,u_1 \in \mathcal E_\eta(Y)$, the latter case cannot happen as $I_\eta(u_0),I_\eta(u_1)>0$. So $I_\eta(u_t)> -\infty$, giving $u_t \in \mathcal E^1_\eta(Y)$.

Since $t \to d_1^\eta(u^j_0,u^j_t)$ is also linear, by taking again a limit, and using Proposition \ref{prop: AM_mon_cont}, it follows that so is $t \to d_1^\eta(u_0,u_t)$. One proves similarly that $t \to d_1^\eta(u_l,u_t)$ is linear for any $l \in [0,1]$, implying that $t \to u_t$ is a $d_1^\eta$-geodesic.
\end{proof}

All that remains is to show that $(\mathcal E^1_\eta(Y),d_1^\eta)$ is complete.  For this we need the following lemma that has independent interest:

\begin{lemma} \label{lemma: d_1monotonone_conv} 
(i) Suppose $u_j \in \mathcal E^1_\eta(Y)$ decreases or increases a.e. to $u \in \mathcal E^1_\eta(Y)$. Then $d_1^\eta(u_j,u) \to 0$.\\
(ii) Suppose $u_j \in \mathcal E^1_\eta(Y)$ is decreasing or increasing a.e. and $d_1^\eta(0,u_j)$ is bounded. Then $u=\lim_j u_j \in \mathcal E^1_\eta(Y)$ and $d_1^\eta(u_j,u) \to 0$.
\end{lemma}
\begin{proof} When $w,v \in \mathcal E^1_\eta(Y)$ satisfies $w \leq v$, observe that $d_1^\eta(w,v)=I_\eta(v)-I_\eta(w)$. By 
Proposition \ref{prop: AM_mon_cont} the result then follows.

For part (ii), as $|I_\eta({u_j})| \leq d_1^\eta(0,u_j)$ it follows that $I_\eta({u_j})$ is bounded.  Since, $I_\eta(u)=\lim_j I_\eta(u_j)$ we get that $u \in \mathcal E^1_\eta(Y)$ and part (i) yields the last part of the conclusion.
\end{proof}
As in the smooth K\"ahler case, the operator $u \to P_\eta(v,u)$ is $d_1^\eta$-contractive:
\begin{proposition} \label{prop: contractivity}Suppose $v,u_0,u_1 \in \mathcal E^1_\eta(Y)$. Then we have $d_1^\eta(P_\eta(v,u_0),P_\eta(v,u_1)) \leq d_1^\eta(u_0,u_1).$
\end{proposition}
\begin{proof} Suppose first that $v,u_0,u_1 \in \PSH_\eta(Y) \cap L^\infty$. By Proposition \ref{thm: da2properties}, we have
$$d_1^{\eta+\varepsilon \eta'}(P_{\eta+\varepsilon \eta'}(v,u_0),P_{\eta+\varepsilon \eta'}(v,u_1)) \leq d_1^{\eta_+\varepsilon \eta'}(u_0,u_1).$$
Letting $\varepsilon \to 0$, using the definition of $I_\eta$ and classical Bedford--Taylor theory, we conclude that inequality of the proposition holds for bounded potentials. For $v,u_0,u_1 \in \mathcal E^1_\eta(Y)$, we have 
\begin{flalign*}
d^\eta_1(P_\eta(\max(v,-j),\max(u_0,-j)),& P_\eta(\max(v,-j),\max(u_1,-j)) ) \\
&\leq d^\eta_1(\max(u_0,-j),\max(u_1,-j)).
\end{flalign*}
Taking the limit $j \to \infty$, by Proposition \ref{prop: AM_mon_cont} the result follows. 
\end{proof} 

\begin{proposition} $(\mathcal E^1_\eta(Y),d^\eta_1)$ is a complete geodesic metric space.\label{prop: E1complete}
\end{proposition}

\begin{proof} We need to argue the completeness. Suppose $\{u_j \}_j \subset \mathcal E^1_\eta(Y)$ is a Cauchy sequence. We can assume without loss of generality that $d_1^\eta(u_j,u_{j+1}) \leq 1/2^j$.

We introduce $v^k_l = P_\eta(u_k,u_{k+1},\ldots,u_{k+l}) \in \mathcal E^1_\eta(Y), \ l,k \in \Bbb N$. We argue first that each decreasing sequence $\{ v^k_l\}_{l \in \Bbb N}$ is $d_1^\eta-$Cauchy. Given our assumptions, this will follow if we show that $d_1^\eta(v^k_{l+1},v^k_l) \leq d_1^\eta (u_{l+1},u_{l}).$
We observe that $v^k_{l+1}=P_\eta(v^k_l,u_{k+ l+1})$ and $v^k_l=P_\eta(v^k_l,u_{k+l})$. Using this and Proposition \ref{prop: contractivity} we can write:
$$d_1^\eta(v^k_{l+1},v^k_l) = d_1^\eta(P_\eta(v^k_l,u_{k+l+1}),P_\eta(v^k_l,u_{k+l})) \leq d_1^\eta(u_{k+l+1}, u_{k+l})\leq \frac{1}{2^{k+l}}.$$

As we have shown in Lemma \ref{lemma: d_1monotonone_conv}, it follows now that each sequence $\{ v^k_l\}_{l \in \Bbb N}$  is $d_1^\eta-$convergening to some $v^k \in \mathcal E^1_\eta(Y)$. Using the same trick as above, one can prove:
$$d_1^\eta(v^k,v^{k+1}) =\lim_{l \to \infty}d_1^\eta(v^k_{l+1},v^{k+1}_l)= \lim_{l \to \infty}d_1^\eta(P_\eta(u_k,v^{k+1}_{l}),P_\eta(u_{k+1},v^{k+1}_l))\leq d_1^\eta (u_k,u_{k+1}),$$
as $d_1^\eta(u_k,u_{k+1}) \leq \frac{1}{2^k}$ it follows that $v^k$ is an increasing $d_1^\eta$-Cauchy sequence, hence its limit $v = \lim_k v^k$ satisfies $v \in \mathcal E^1_\eta(Y)$ (Lemma \ref{lemma: d_1monotonone_conv}). Using Proposition \ref{prop: contractivity}, we obtain the following estimates, confirming that $v^k$ and $u_k$ are equivalent $d_1^\eta$-Cauchy sequences:
\begin{flalign*}
d_1^\eta(v^k,u_k) &=\lim_{l \to \infty}d_1^\eta(v^k_l,u_k)=\lim_{l \to \infty}d_1^\eta((P_\eta(u_k,v^{k+1}_{l-1}),P_\eta(u_k,u_k))\\
&\leq\lim_{l \to \infty}d_1^\eta(v^{k+1}_{l-1},u_k)=\lim_{l \to \infty}d_1^\eta(P_\eta(u_{k+1},v^{k+2}_{l-2}),u_k)\\
&\leq \lim_{l \to \infty}d_1^\eta(P_\eta(u_{k+1},v^{k+2}_{l-2}),u_{k+1}) + d_1^\eta(u_{k+1},u_k)\\
&\leq \lim_{l \to \infty} \sum_{j=k}^{l+k}d_1^\eta(u_j,u_{j+1}) \leq \frac{1}{2^{k-1}}.
\end{flalign*}
As $v^k \to_{d^1_\eta} v \in \mathcal E_\eta^1(Y)$, the result follows.
\end{proof}

With the above theorem in place, we can now follow exactly the same arguments as in  \cite[Theorem 3, Theorem 5(ii), Proposition 5.9]{da2} to prove the following result with multiple points, giving additional properties of $d_1$-convergence:

\begin{theorem}\label{thm: d_1convergencethm}(i)There exists $C=C(n)>1$ such that 
\begin{equation}\label{eq: d_1growthchar}
\frac{1}{C}d_1^\eta(u_0,u_1) \leq \int_Y |u_1 - u_0| \eta_{u_0}^n + \int_Y |u_1 - u_0| \eta_{u_1}^n \leq Cd_1^\eta(u_0,u_1), \ \ u_0,u_1 \in \mathcal E^1_\eta(Y).
\end{equation}
(ii) For any $u_j,u \in \mathcal E^1_\eta(Y)$ we have $d_1^\eta(u_j,u) \to 0$ if and only if $\int_Y |u_j-u|\eta^n \to 0$ and $I_\eta(u_j) \to I_\eta(u)$. \\
(iii) If $d_1^\eta(u_j,u) \to 0$, then $\int_X |u_j - u|\eta_v^n \to 0$ for any $v \in \mathcal E^1_\eta(Y)$.
\end{theorem}

Finally, let us relate all of the above to the geometry of the space $\mathcal H_\o(X)$ and $\mathcal E^1_\o(X)$ respectively. As noted in Section 3.1, using the pullback $\pi^*$ we can identify $\mathcal E^1_\o(X)$ with $\mathcal E^1_\eta(Y)$. Using this, one can introduce a complete geodesic metric space structure on $\mathcal E^1_\o(X)$ via pullback of geometric data. Namely $d_\o^1:\mathcal E^1_\o(X) \times \mathcal E^1_\o(X) \to \Bbb R$ is introduced as 
\begin{equation}\label{eq: d_1_odef}
d_\o^1(u_0,u_1):=d_\eta^1(\pi^*u_0,\pi^* u_1).
\end{equation}
Also, given $u_0,u_1\in \mathcal E^1_
\o(X)$ the finite energy geodesic $[0,1] \ni t \to u_t \in \mathcal E^1_\o(X)$ connecting $u_0,u_1$ is just the curve for which $\pi^* u_t = v_t$, where $[0,1] \ni t \to v_t \in \mathcal E^1_\eta(Y)$ is the finite energy geodesic connecting $\pi^* u_0,\pi^* u_1$. Finally, we arrive at the main theorem of this section, confirming that the choice of resolution $\pi:Y \to X$ did not play a role in our constructions above:
\begin{theorem}\label{thm: mainmetricthm_singular} Given a normal K\"ahler space $(X,\o)$ with Hodge metric $\o$, the complete geodesic metric space $(\mathcal E^1_
\o(X),d^\o_1)$ and its finite energy geodesics defined above are independent of the desingularization $\pi: Y \to X$. Additionally, $d_1^\o$ satisfies the following explicit formula:
\begin{equation}\label{eq: d1o_formula}
d_1^\o(u_0,u_1)= I_\o(u_0)+I_\o(u_1)-2I_\o(P_\o(u_0,u_1)).\end{equation}
Finally $\overline{(\mathcal H_\o(X),d_1^\o)} = (\mathcal E^1_
\o(X),d^\o_1)$. 
\end{theorem}
\begin{proof}As $\pi^*$ gives a bijection between $\mathcal E^1_\o(X)$ and $\mathcal E^1_\eta(Y)$ it follows that $(\mathcal E^1_\o(X),d_1^\o)$ is complete, as defined in \eqref{eq: d_1_odef}. We argue that $d_1^\o$ is independent of $\pi$, i.e., it satisfies \eqref{eq: d1o_formula}. 

First we note that $\pi^* P_\o(u_0,u_1)=P_\eta(\pi^* u_0,\pi^* u_1)$. Indeed, if $v \in \PSH_\o(X)$ such that $v \leq u_0,u_1$, then $\pi^* v \in \PSH_\eta(Y)$ satisfies  $\pi^* v \leq \pi^* u_0, \pi^* u_1$, proving that $\pi^* P_\o(u_0,u_1)\leq P_\eta(\pi^* u_0,\pi^* u_1)$. For the other inequality, let $v:=P_\eta(\pi^* u_0,\pi^* u_1)$. Clearly, $v \leq \pi^* u_0, \pi^* u_1$ and on $X_{reg}$ we have ${\pi^{-1}}^*v \leq u_0,u_1$. Using \cite[Theorem 1.7]{de1} we get that ${\pi^{-1}}^*v$ extends to a function $h \in \PSH_\o(X)$ satisfying $h \leq u_0,u_1$, hence $h \leq P_\o(u_0,u_1)$. Applying $\pi^*$ to this inequality we obtain that $\pi^* P_\o(u_0,u_1)\geq \pi^* h = P_\eta(\pi^* u_0,\pi^* u_1)$. 

Let $u_0,u_1 \in \mathcal E^1_\o(X)$.
Using similar ideas one can show that the finite energy geodesic $[0,1] \ni t \to u_t \in \mathcal E^1$ joining $u_0,u_1$ defined above does not depend on $\pi$. Indeed, the following expected formula can be given:
\begin{equation}
u = \sup \{v \in \PSH_{\textup{pr}_2^*\o}(S \times X) \textup{ s.t. } v(t+ir,y)=v(t,y) \textup{ and } v_0 \le u_0, v_1 \leq u_1 \}.
\end{equation}
By \eqref{eq: d_1_odef} we can continue to write:
$$d_1^\o(u_0,u_1)=d_1^\eta(\pi^*u_0,\pi^*u_1) = I_\eta(\pi^* u_0)+I_\eta( \pi^* u_1)-2I_\eta(P_\eta(\pi^* u_0,\pi^* u_1)).$$
Above we argued that $P_\eta(\pi^* u_0,\pi^* u_1)= \pi^*P_\o(u_0,u_1)$ and since $I_\o(\cdot)=I_\eta(\pi^*(\cdot))$, \eqref{eq: d1o_formula} follows.

Finally, $\overline{(\mathcal H_
\o(X),d^\o_1)}=(\mathcal E^1_
\o(X),d^\o_1)$ follows from Theorem \ref{thm: EGZapprox} and Lemma \ref{lemma: d_1monotonone_conv}.
\end{proof}

Due to this last result, in general we will not distinguish between the complete geodesic metric spaces $(\mathcal E^1_\o(X),d_1^\o)$ and $(\mathcal E^1_\eta(Y),d_1^\eta)$. A similar point of view is also present in \cite{bbegz} and we hope this will not be a source of confusion. 

Let $\textup{Aut}_0(X,J)$ be the identity component of the group of holomorphic automorphisms of $X$. Unfortunately, this group does not lift ``efficiently" to the desingularization $Y$. Still, as in the K\"ahler case, this group acts by $d_1^\o$-isometries on the normalization of $\mathcal E^1_\o(X)$:
\begin{proposition}\label{prop: Gd1isometry}
$G:=\textup{Aut}_0(X,J)$ acts by $d_1^\o$-isometries on $\mathcal E^1_\o(X) \cap I_\o^{-1}(0)$.
\end{proposition}
\begin{proof}We argue that for $u \in \mathcal H_
\o(X) \cap I_\o^{-1}(0)$ we have 
\begin{equation}\label{eq: pullbackformula}
g.u = h_g + u \circ g \in \mathcal H_\o(X) \cap I_\o^{-1}(0)
\end{equation} 
for any $g \in G$, where $h_g \in \mathcal H_\o(X)\cap I_\o^{-1}(0)$ is the potential for which $g^* \o = \o_{h_g}$. Indeed, first we notice that we have $g^* \o = \o + i\ddbar (h_g + u \circ g)$, hence \eqref{eq: pullbackformula} holds up to a constant. As $I_\o(h_g)=0$, using 
\eqref{eq: AM_diff}, the next line of calculations implies $I_\o(h_g + u \circ g)=0$, finishing the proof of \eqref{eq: pullbackformula}:
$$I_\o(h_g + u \circ g) -I_\o(h_g)= \frac{1}{(n+1)V}\sum_{j=0}^n\int_{X_{reg}} u \circ g \o_{g.u}^{j}\wedge \o_{h_g}^{n-j}=I_\o(u)-I_\o(0)=0.$$
Suppose $u_0,u_1 \in \mathcal H_\o(X)\cap I_\o^{-1}(0)$. From \eqref{eq: pullbackformula} it follows that 
\begin{equation}
P_\o(g.u_0,g.u_1) = g. P_\o(u_0,u_1)=h_g + P_\o(u_0,u_1) \circ g.
\end{equation}
Using this and \eqref{eq: AM_diff} we can write
\begin{flalign*}
I_\o(g.u_0)&-I_\o(P_\o(g.u_0.g.u_1))= I_\o(g.u_0)-I_\o(g.P_\o(u_0,u_1)) \\
&=\frac{1}{(n+1)V}\sum_{j=0}^n\int_{X_{reg}}(u_0-P_\o(u_0,u_1))\circ g \  \o_{g.u}^j \wedge \o_{P_\o(g.u_0.g.u_1)}^{n-j}.
\end{flalign*}
As the restriction of $g$ to $X_{reg}$ is an automorphism, a change of variables gives that this latter sum of integrals is equal to $I_\o(u_0)-I_\o(P_\o(u_0,u_1))$. As a consequence,  $d_1(g.u_0,g.u_1)=d_1(u_0,u_1)$ for $u_0,u_1 \in \mathcal H_\o(X)$, hence the action of $G$ on $\mathcal H_\o(X)$ does induce a $d_1$-isometry. As $\mathcal H_\o(X)$ is $d_1$-dense inside $\mathcal E^1_\o(X)$ it follows that this action extends to a $d_1$-isometric action on $\mathcal E^1_\o(X)$.
\end{proof}

\begin{lemma}Suppose $K \subset \textup{Aut}_0(X,J)$ is a compact subgroup with Lie algebra $\mathfrak k$. Let $u_0 \in \mathcal H_\o^0(X):=\mathcal H_\o(X) \cap I^{-1}(0)$ be such that $\o_{u_0}$ is $K$-invariant. For any $V \in J \mathfrak k$ the induced 1-parameter group of automorphisms $t \to f_t$ gives a  $d_1$-geodesic $[0,\infty) \ni t \to u_t \in \mathcal H_\o(X) \cap I_\o^{-1}(0)$, satisfyng $f^*_t{\o_{u_0}}=\o_{u_t}$, whose speed depends continuously on $V$.\label{lemma: reductivitycor}
\end{lemma}
\begin{proof} Let $u \in \PSH_{\textup{pr}_2^*\o}(\Bbb C^+ \times X)$ be defined by $u(s,x)= u_{\textup{Re }s}(x)$. We first note that $X_{reg}$ stays invariant under elements of $\textup{Aut}_0(X,J)$, in particular $f_t(X_{reg})=X_{reg}$. Hence the same argument   as in the smooth case \cite{mab} gives that on $\Bbb C^+ \times X_{reg}$ we have 
$$(\textup{pr}_2^* \o + i\ddbar u)^{n+1}=0,$$
giving that $[0,\infty) \ni t \to u_t \in \mathcal H_\o(X)\cap I_\o^{-1}(0)$ is a $d_1$-geodesic ray.

We turn to the last statement. The correspondence $V \to f_1$ is smooth, hence using the identity $f_1^* \o =\o_{u_1}$, so is the correspondence $V \to u_1$. Finally, this implies that the correspondence $V \to I_\o(u_1)$ is smooth and $V \to I_\o(P_\o(u_0,u_1))$ is at least continuous, giving that $V \to d_1(u_0,u_1)=I_\o(u_0)+I_\o(u_1)-2I_{\o}(P_\o(u_0,u_1))$ is also at least continuous. 
\end{proof}

\section{The existence/properness theorems for log Fano pairs}
\label{sec: existence/properness thm}
For a log Fano pair $(X,D)$, let us recall the definition of the K-energy and Ding functional from \cite[Section 4.1]{bbegz}. 
Let $h$ be the smooth metric on $-K_X-D$ such that $\Theta(h)=\o$. Let $p \in \Bbb N$ such that $p(K_X +D)$ is Cartier on $X$. If $\sigma$ is a non-vanishing section of the line bundle associated 
to $p(K_X + D)$ on an open subset $U$ of $X_{reg}$, one can introduce the following measure on $U$:
$$\mu = V \frac{(i^{pn}\sigma \wedge \overline{\sigma})^{1/p}}{ h^p(\sigma,\sigma)^{1/p}}.$$
One can see that $\mu$ is independent of the choice of $\sigma$ and $U$, and as such can be canonically extended to $X_{reg}$. We can now introduce the K-energy $\mathcal K_D : \mathcal E^1_\o(X) \to (-\infty,\infty]$ and the Ding energy $\mathcal F_D : \mathcal E^1_\o(X) \to \Bbb R$ by
$$\mathcal K_D(u) = \frac{1}{V}\int_{X_{reg}} \log \Big(\frac{\o^n_{u}}{\mu}\Big)-I_\o(u) + \frac{1}{V}\int_{X_{reg}} u \o_{ u}^n$$.
$$\mathcal F_D(u) = -I_\o(u) - \log \int_{X_{reg}} e^{-u}\mu.$$
It is well known that $\mathcal F_D$ is $d_1$-continuous and $\mathcal K_D$ is only $d_1$-lsc (Theorem \ref{thm: d_1convergencethm}(ii) and  \cite[Lemma 4.3]{bbegz}). Also, by \cite[Lemma 4.4(i)]{bbegz} the following estimate holds:
\begin{equation}\label{eq: FDKD_ineq}
\mathcal F_D \leq \mathcal K_D. 
\end{equation}
Recall the $J$ functional $J : \mathcal E^1_\o(X) \to \Bbb R$, defined by $J(u)= \int_{X_{reg}} u \o^n - I_\o(u).$
The same proof as in \cite[Proposition 5.5]{dr2}  gives the existence of $C(n)>1$ such that for $u \in \mathcal H^0_\o = \mathcal H_\o(X) \cap I_\o^{-1}(0)$ we have
\begin{equation}\label{eq: Jd_1eqv}
\frac{1}{C}J(u)-C \leq d_1^\o(0,u)\leq C J(u) + C.
\end{equation}

\subsection{The proof of Theorem \ref{thm: propernesstheorem}}

First notice that Theorem \ref{thm: metrictheorem} applies in this setting, as $[\omega] \in c_1(-K_X - D)$, hence a multiple of $\omega$ is a Hodge metric. We will verify the conditions of Theorem \ref{thm: ExistencePrinc} for the data $(\mathcal R,d,F,G)$, where $\mathcal R=\PSH_\o(X) \cap L^\infty \cap I_\o^{-1}(0)$, $d=d_1^{\o}$, $F=\mathcal F_D$ and $G= Aut_0(X,D)$. Given that $\mathcal F_D$ is $d^\o_1$-continuous and $\overline{\PSH_\o(X) \cap L^\infty\cap I_\o^{-1}(0)}^{d_1^\o}=\mathcal E^1_\o(X)\cap I_\o^{-1}(0)$, the  conditions \ref{a1}-\ref{a4} are verified. 

We now focus on the conditions \ref{p1}-\ref{p7}. Condition \ref{p1} is a theorem of Berndtsson \cite{brm1} adapted to the singular setting in \cite[Theorem 11.1]{bbegz}. Condition \ref{p2} is proved in exactly the same way as in \cite[Proposition 5.27]{dr2} that treats the case when $(X,\o)$ is smooth K\"ahler and $D$ is just a smooth connected divisor. By \cite[Theorem 4.8]{bbegz}, we have that any $\mathcal E^1_\o(X)$-minimizer of $\mathcal F_D$ is actually bounded K\"ahler--Einstein potential, hence \ref{p3} holds. Condition \ref{p4} was verified in Proposition \ref{prop: Gd1isometry}. Condition \ref{p5} is just \cite[Theorem 5.1]{bbegz}.

When a K\"ahler--Einstein metric exists, by \cite[Corollary 5.2]{bbegz} the group $G$ is reductive, hence the conditions of Lemma \ref{lemma: reductivitycor} are satisfied for $K$, the maximal compact subgroup of $G$ and a K\"ahler--Einstein potential $u_0$. This result in turn implies that the conditions of \cite[Proposition 6.8]{dr2} are satisfied, hence condition \ref{p6} holds. 

Condition \ref{p7} is classical  for $u_0,u_1 \in \mathcal H_{\o}(X)\cap I^{-1}_\o(0)$, and extends to $u_0,u_1 \in \PSH_{\o}(X)\cap L^\infty \cap I^{-1}_\o(0)$ by density. Finally,  Theorem \ref{thm: ExistencePrinc} and gives that a K\"ahler--Einstein potential inside $\PSH_{\o}(X)\cap L^\infty \cap I^{-1}_\o(0)$ exists if and only if $\mathcal F_D$ is $G$-invariant and there exists $C,B>0$ such that
$$\mathcal F_D(U) \geq Cd^\o_{1G}(G0,Gu) -B, \ \ u \in \PSH_\o(X) \cap L^\infty \cap I_\o^{-1}(0).$$
Since  $\mathcal H_\o^0(X)=\mathcal H_\o(X)\cap I^{-1}_\o(0)$ is $d_1$-dense in $\PSH_\o(X) \cap L^\infty \cap I_\o^{-1}(0)$, and we have the double estimate of \eqref{eq: Jd_1eqv}, the proof of Theorem \ref{thm: propernesstheorem} is complete for $F=\mathcal F_D$. 

In case $F=\mathcal K_D$, one has to proceed differently, because the proof of the convexity of the K-energy does not carry over directly from \cite{bb} to our more singular setting, though this is still expected to be true. Instead, by \eqref{eq: FDKD_ineq}, and the above arguments for $\mathcal F_D$, we get the implication (existence of KE) $\to$ (properness of $\mathcal K_D$). For the reverse direction assume that we have
$$\mathcal K_D(u) \geq C J_{G}(Gu)-B \geq C' d_{1G}^\o(G0,Gu) - B', \ \ u \in \mathcal H_\o^0(X).$$
Hence $\mathcal K_D$ is bounded from below and there exists a $d_1$-bounded sequence $u_j \in \mathcal H_\o^0(X)$ such that $\mathcal K_D(u_j) \to \inf_{u \in \mathcal E^1_\o(X)}\mathcal K_D(u)$. By \cite[Theorem 2.17]{bbegz} the sequence $u_j$ $d^\o_1$-subconverges to $u \in \mathcal E^1_\o(X)$ that minimizes $\mathcal K_D$. Finally \cite[Theorem 4.8]{bbegz} implies that $u$ is a K\"ahler--Einstein potential.

\subsection{The proof of Theorem \ref{thm: propernessthm_soliton}}

For the precise definition of the functionals $\mathcal K_V$ and $\mathcal F_V$ we refer to \cite{bwn}. The proof goes in exactly the same spirit as that of  Theorem \ref{thm: propernesstheorem}.

The main point is to verify the conditions of Theorem \ref{thm: ExistencePrinc}
for the data $(\mathcal R,d,F,G)$, where $\mathcal R=\PSH_\o^T(X) \cap L^\infty \cap I^{-1}_\o(0)$, $d=d_1^{\o}$, $F=\mathcal F_V$ and $G= Aut_0(X,V)$.  This is carried out exactly the same way as in the proof of the previous theorem, by replacing the ingredients from \cite{bbegz} with the corresponding ones from \cite{bwn}. To establish \ref{p1} we can use \cite[Propostion 3.1]{bwn}. Condition \ref{p2} is verified the same way as in the smooth case \cite[Proposition 5.29]{dr2}. Condition \ref{p3} is verified in \cite[Theorem 3.3]{bwn}. Condition \ref{p4} is verified using the soliton analog of Proposition \ref{prop: Gd1isometry}. Condition \ref{p5} is just  \cite[Theorem 3.6]{bwn}. Condition \ref{p6} again follows from the reductivity of $G$ in case a K\"ahler-Ricci soliton exists \cite[Corollary 3.7]{bwn}, and an application of \cite[Proposition 6.8]{dr2}. \ref{p7} is again standard.

\paragraph{Acknowledgments.} As we wanted to discuss as quickly as possible the relationship between energy properness and existence of special K\"ahler metrics on singular varieties in Theorem \ref{thm: propernesstheorem} and Theorem \ref{thm: propernessthm_soliton}, we took a very economical approach in establishing Theorem \ref{thm: metrictheorem}. This was possible due to the special characteristics of $L^1$-Mabuchi geometry, but as a result, issues related to regularity of weak geodesic segments where completely avoided. Without addressing these questions however, the more general $L^p$-Mabuchi geometries for singular classes likely cannot be explored. This is exactly the approach of the independent work by Di Nezza--Guedj that also discusses applications to canonical K\"ahler metrics \cite{DNG}. 

After the first version of the paper appeared, Jeff Streets informed us that that the main approximation result of \cite{egz2} had an incomplete proof (see also \cite{str,egz3}). Consequently, in the present paper the approximation result of \cite{cgz} has to be used instead that only applies in the case of integral K\"ahler classes (see Theorem \ref{thm: EGZapprox}). As a result, in this final version of the paper we can only prove Theorem \ref{thm: metrictheorem} for integral K\"ahler classes, though this more modest result is still enough to establish all the applications related to canonical K\"ahler metrics (Theorems 2.2-2.4).
 
We would like to thank Eleonora Di Nezza, Henri Guenancia, Swarnava Mukhopadhyay and Yanir Rubinstein for many helpful discussions. An initial version of this paper was circulated among experts in early 2016. 
The finishing touches where added in April 2016, when the author was a research member at MSRI, and was partially supported there by NSF grant DMS-1440140. The author's research is also partially supported by BSF grant 2012236 and NSF grant DMS-1610202.

\let\OLDthebibliography\thebibliography 
\renewcommand\thebibliography[1]{
  \OLDthebibliography{#1}
  \setlength{\parskip}{1pt}
  \setlength{\itemsep}{1pt plus 0.3ex}
}

\begin{small}
{\sc University of Maryland}

\noindent {\tt tdarvas@math.umd.edu}
\end{small}
\end{document}